\documentclass[12pt]{article}
\usepackage{epsfig}
\usepackage{enumerate}
\usepackage{xcolor}

\usepackage{amssymb}
\usepackage{amsmath}
\usepackage{amsfonts}
\usepackage{amsthm}
\newtheorem{theorem}{Theorem}
\newtheorem{lemma}[theorem]{Lemma}

\newtheorem{corollary}[theorem]{Corollary}
\newtheorem{conjecture}[theorem]{Conjecture}

\newtheoremstyle{problem-def}%
{}
{}
{}
{}
{\sc}
{.}
{.5em}
{\thmnote{#3}}
\theoremstyle{problem-def}
\newtheorem*{problem-def}{Problem}

\def\df#1{{\em #1\/}}

\def\A{\mathcal{A}}
\def\B{\mathcal{B}}
\def\C{\mathcal{C}}

\def\F{\mathcal{F}}

\def\NN{\mathbb{N}}

\newcounter{cases}
\newcounter{subcases}[cases]

\newenvironment{casesblock}{%
\setcounter{cases}{0}\setcounter{subcases}{0}}{}

\def\case#1{\stepcounter{cases}\medskip\noindent{\bf Case \arabic{cases}: }#1.\par}
\def\subcase#1{\stepcounter{subcases}\medskip\noindent{\bf Subcase \roman{subcases}: }#1.\par}

\newcommand\point[1]{\par\medskip\noindent{\rm #1}}

\def\ifempty#1#2#3{%
\def\ifemptytest{#1}%
\def\ifemptyempty{}%
\ifx\ifemptyempty\ifemptytest #2\else #3\fi%
}

\def\imp{\Rightarrow}
\def\sm{\setminus}

\def\gph#1{{\sc #1}}

\def\eg{\widehat{g}}

\def\trans{\tau}

\def\minor{\ast}
\def\hat{\widehat}

\def\lab{\lambda}
\def\Forb{\text{Forb}}

\def\mid{\colon\,}

\def\SSS{\mathbb{S}}
\def\sss{\subseteq}

\begin{document}
\title{Obstructions for two-vertex alternating embeddings of graphs in surfaces}
\author{%
  Bojan Mohar
  \thanks{Supported in part by an NSERC Discovery Grant (Canada),
    by the Canada Research Chair program, and by the
    Research Grant P1--0297 of ARRS (Slovenia).}~\thanks{On leave from:
    IMFM \& FMF, Department of Mathematics, University of Ljubljana, Ljubljana,
    Slovenia.}
  \qquad
  Petr \v{S}koda\\\\
  \normalsize
  Department of Mathematics, \\
  \normalsize
  Simon Fraser University,\\
  \normalsize
  8888 University Drive,\\ 
  \normalsize
  Burnaby, BC, Canada.}
\date{}
\maketitle
\begin{abstract}
A class of graphs that lies strictly between the classes of graphs of genus (at most) $k-1$ and $k$ is studied.
For a fixed orientable surface $\SSS_k$ of genus $k$, let $\A_{xy}^k$  be the minor-closed class of graphs with terminals $x$ and $y$ that either embed into $\SSS_{k-1}$ or
admit an embedding $\Pi$ into $\SSS_k$ such that there is a $\Pi$-face where $x$ and $y$ appear twice in the alternating order.
In this paper, the obstructions for the classes $\A_{xy}^k$ are studied.
In particular, the complete list of obstructions for $\A_{xy}^1$ is presented.
\end{abstract}

\section{Introduction}

For a simple graph $G$, let $g(G)$ be the \df{genus} of $G$, that is, the minimum $k$ such that $G$ embeds into the orientable surface $\SSS_k$.
Similarly $\eg(G)$ stands for the \df{Euler genus} of $G$. A \df{combinatorial embedding} $\Pi$ of $G$ is a pair $(\pi, \lambda)$ where $\pi$ assigns each vertex $v \in V(G)$
a cyclic permutation of edges adjacent to $v$ called the \df{local rotation} around $v$ and the function $\lambda: E(G) \to \{-1, 1\}$ describes the signature
of edges when $\Pi$ is non-orientable. A \df{$\Pi$-face} is a walk in $G$ around a face of $\Pi$ (for a formal definition see for example~\cite{mohar-book}).
Vertices $v_1, \ldots, v_k$ are \df{$\Pi$-cofacial} if there is a $\Pi$-face where the vertices $v_1, \ldots, v_k$ appear in some order.

For an edge $e$ of $G$, the two standard graph operations, \df{deletion of $e$}, $G - e$, and \df{contraction of $e$}, $G/e$, are called \df{minor operations}
and are denoted by $G \minor e$ when no distinction is neccessary. A graph $H$ is a \df{minor} of $G$ if $H$ is obtained from a subgraph of $G$ by a sequence
of minor operations.
A family of graphs $\C$ is \df{minor-closed} if, for each graph $G \in \C$, all minors of $G$ belong to $\C$.
A graph $G$ is a (\df{minimal}\/) \df{obstruction} for a family $\C$ if $G$ does not belong to $\C$ but for every edge $e$ of $G$,
both $G - e$ and $G / e$ belong to $\C$. The well-known result of Robertson and Seymour~\cite{robertson-2004} asserts that the list of obstructions
is finite for every minor-closed family of graphs.

For a fixed surface $\SSS_k$, the graphs that embed into $\SSS_k$ form a minor-closed family and it is of general interest to understand
the obstructions $\Forb(\SSS_k)$ for these families. Unfortunately, $\Forb(\SSS_1)$ already contains thousands of graphs and is not yet determined~\cite{gagarin-2009}.
We approach the problem by studying graphs in $\Forb(\SSS_k)$ of small connectivity (see~\cite{mohar-low}).

In this paper we study a phenomenon that arises when joining two graphs by two vertices.
Given graphs $G_1$ and $G_2$ such that $V(G_1) \cap V(G_2) = \{x, y\}$,
the union of $G_1$ and $G_2$, that is the graph $(V(G_1) \cup V(G_2), E(G_1) \cup E(G_2))$, is an \df{$xy$-sum} of $G_1$ and $G_2$ 
(or a \df{$2$-sum} if the vertices are not important).
Sometimes, we also call $G$ to be an $xy$-sum of $G_1$ and $G_2$ even if the edge $xy$ is an edge of $G_1$ or $G_2$ but is not present in $G$.
To determine the genus of the $xy$-sum of $G_1$ and $G_2$, it is neccessary to know if $G_1$ (and $G_2$) has a minimum
genus embedding $\Pi$ such that there is a $\Pi$-face in which $x$ and $y$ appear twice in the alternating order (see~\cite{decker-1981,decker-1985}).
For vertices $x, y \in V(G)$, we say that $G$ is \df{$xy$-alternating on $\SSS_k$} if $g(G) = k$ and $G$ has an embedding $\Pi$ of genus $k$
with a $\Pi$-face $W = v_1\ldots v_l$ and indices $i_1, \ldots, i_4$ such that $1 \le i_1 < i_2 < i_3 < i_4 \le l$, 
$v_{i_1} = v_{i_3} = x$, and $v_{i_2} = v_{i_4} = y$.

A graph $G$ is \df{$k$-connected} if $G$ has at least $k+1$ vertices and $G$ remains connected after deletion of any $k-1$ vertices.
A graph has \df{connectivity} $k$ if it is $k$-connected but not $(k+1)$-connected.

To determine minimal obstructions of connectivity 2, we need to know which graphs are minimal {\em not\/} $xy$-alternating (see~\cite{mohar-low}).
For $k \ge 1$, let $\A_{xy}^k$ be the class of graphs with \df{terminals} $x$ and $y$ 
that are either embeddable in $\SSS_{k-1}$ or are $xy$-alternating on $\SSS_k$.
When performing minor operations on \df{graphs with terminals}, we do not allow a contraction identifying two terminals to a single vertex.
Also, when contracting an edge joining a terminal and a non-terminal vertex, the new vertex is a terminal. Thus the number of terminals of a minor is the same as of the original graph.
A homomorphism of two graphs with terminals is an isomorphism if it is a graph isomorphism and (non-)terminals are mapped onto (non-)terminals.
In particular, automorphisms that switch the terminals are considered.
Under these restrictions, 
$\A_{xy}^k$ is a minor-closed family of graphs with two terminals. 
Let $\F_{xy}^k$ be the set of minimal obstructions for $\A_{xy}^k$, that is, a graph $G$ belongs to $\F_{xy}^k$ if $G \not\in \A_{xy}^k$ and, for each edge $e \in E(G)$
and each allowed minor operation $\minor$, $G \minor e \in \A_{xy}^k$.
It is shown in Sect.~\ref{sc-general} that $\F_{xy}^k$ is finite for each $k \ge 1$.
Note that each vertex of a graph in $\F_{xy}^k$ has degree at least 3 unless it is a terminal.

A \df{Kuratowski graph} is a graph isomorphic to $K_5$, the complete graph on five vertices, or to $K_{3,3}$, the complete bipartite graph on three and three vertices. 
For a fixed Kuratowski graph $K$, a \df{Kuratowski subgraph} in $G$ is a minimal subgraph of $G$ that contains $K$ as a minor.
A \df{K-graph} $L$ in $G$ is a subdivision of $K_4$ or $K_{2,3}$ that can be extended to a Kuratowski subgraph in $G$.
We are using extensively the following well-known theorem.

\begin{theorem}[Kuratowski~\cite{kuratowski-1930}]
\label{th-kuratowski}
  A graph is planar if and only if it does {\em not} contain a Kuratowski subgraph.
\end{theorem}

We also use the following classical theorem (see~\cite[Theorem~6.3.1]{mohar-book}).

\begin{theorem}
  \label{th-disk-ext}
  Let $G$ be a connected graph and $C$ a cycle in $G$. 
  Let $G'$ be a graph obtained from $G$ by adding a new vertex joined to all vertices of $C$.
  Then $G$ can be embedded in plane with $C$ as an outer cycle unless $G$ contains
  an obstruction of the following type:
  \begin{enumerate}[\rm(a)]
  \item 
    a pair of disjoint crossing paths,
  \item
    a tripod, or
  \item
    a Kuratowski subgraph contained in a 3-connected block of $G'$ distinct from 
    the 3-connected block of $G'$ containing $C$.
  \end{enumerate}
\end{theorem}


Let $G$ be a 2-connected graph. Each vertex of degree different from 2 is a \df{branch vertex}.
A \df{branch} of $G$ is a path in $G$ whose endvertices are branch vertices and such that each
intermediate vertex has degree 2.

Let $H$ be a subgraph of $G$. An \df{$H$-bridge} in $G$ is a subgraph of $G$ which is either an edge not in $H$
but with both ends in $H$, or a connected component of $G - V(H)$ together with all edges which have one end in this 
component and the other end in $H$. For a $H$-bridge $B$, the \df{interior} of $B$, $B^\circ$, is the set $E(B) \cup (V(B) \sm V(H))$ 
containing the edges of $B$ and the vertices inside $B$.
Thus, $G - B^\circ$ is the graph obtained from $G$ by deleting $B$.

Let $B$ be an $H$-bridge in $G$. The vertices in $V(B) \cap V(H)$ are called \df{attachments} of $B$.
The bridge $B$ is a \df{local bridge} if all attachments of $B$ lie on a single branch of $H$.

Let $C$ be a cycle of a fixed orientation and $u$ and  $v$ two vertices in $C$.
The \df{segment} $C[u, v]$ is the path $P$ in $C$ from $u$ to $v$ (in the given orientation of $C$).
Similarly, $C(u, v)$ denotes $P$ without the endvertices and any combination of brackets can be used
to indicate which endvertices are included in the path.
Let $P$ be a segment of $C$ and $B$ a $C$-bridge whose attachments are contained in $P$. 
The \df{support} of $B$ in $P$ is the smallest subsegment of $P$ that contains all attachments of $B$.

For a cycle $C$, two $C$-bridges \df{avoid} each other if there are vertices $u$ and $v$ such that
all attachments of one bridge lie on $C[u,v]$ and all attachments of the other bridge lie on $C[v,u]$.
Otherwise, they \df{overlap}. A $C$-bridge $B$ is \df{planar} if $C \cup B$ is planar.
Let $\B$ be a set of $C$-bridges. The \df{bridge-overlap graph} of $\B$ has vertex set $\B$ and two
bridges are adjacent if they overlap. We use the following well-known theorem.

\begin{theorem}
\label{th-bridge-overlap}
Let $G$ be a graph that consists of a cycle $C$ and a set $\B$ of planar $C$-bridges.
Then $G$ is planar if and only if the bridge-overlap graph of $\B$ is bipartite.
\end{theorem}

The paper is organized as follows. In Sec.~\ref{sc-general}, we study the classes $\F_{xy}^k$ in general.
The rest of the paper is focused on the class $\F_{xy}^1$. A basic classification of $\F_{xy}^1$ is shown
in Sec.~\ref{sc-basic} and the complete list of $\F_{xy}^1$ is provided in the subsequent chapters.
The paper is concluded in Sec.~\ref{sc-main} where the main theorem is proven.

\section{General properties}
\label{sc-general}

In this section we present some general results about graphs in $\F_{xy}^k$, where $k\ge 1$.
In the following, $G /xy$ is the underlying simple graph of the multigraph obtained by identifying vertices $x$ and $y$. 
Note that the edge $xy$ does {\em not} have to be present and, if $xy \in E(G)$, we delete $xy$ before identifying $x$ and $y$.
Let $v_{xy}$ be the vertex obtained after the identification.
For a graph $G$ with terminals $x$ and $y$, let $G^+$ denote the graph $G + xy$ if $xy \not\in E(G)$ and the graph $G$ otherwise.
We will use the  following lemma (see~\cite[Prop.~6.1.2.]{mohar-book}).

\begin{lemma}
  \label{lm-planar-patch}
  Let $G$ be an $xy$-sum of graphs $G_1$ and $G_2$. If $G_2^+$ is planar, then
  each embedding of $G_1^+$ into a surface can be extended to an embedding of $G$
  into the same surface.
\end{lemma}

\begin{proof}
  Since $G_2^+$ is planar, there is a planar embedding of $G_2$ such that $x$ and $y$ are on the infinite face.
 A given embedding $\Pi$ of $G_1^+$ can be extended into the embedding of $G$ by embedding $G_2$
 into a $\Pi$-face incident with the edge $xy$.
\end{proof}

In the sequel, we shall use another graph $G^*$ obtained from a given graph $G$ with given terminals $x$ and $y$.
The graph $G^*$ is obtained as an $xy$-sum of $G$ and $K_5 - xy$ (the graph obtained from $K_5$ with two terminals $x$ and $y$ by deleting the edge $xy$).
We will use a characterization of $xy$-alternating graphs by Decker et al.~\cite{decker-1985}, that a graph $G$ with terminals $x$ and $y$ is $xy$-alternating if and only if $g(G^*) = g(G)$.
They also proved the following theorem:

\begin{theorem}[Decker, Glover, and Huneke~\cite{decker-1985}]
  \label{th-decker-ori}
  If $G$ is an $xy$-sum of graphs $G_1$ and $G_2$, then
  \begin{equation}
    g(G) = \min\{g(G_1^+) + g(G_2^+) - \epsilon(G_1)\epsilon(G_2), g(G_1) + g(G_2) + 1\} \label{eq-ori}
  \end{equation}
  where $\epsilon(G) = 1$ if $G^+$ is $xy$-alternating and $\epsilon(G) = 0$ otherwise.
\end{theorem}

\begin{figure}
  \centering
  \includegraphics{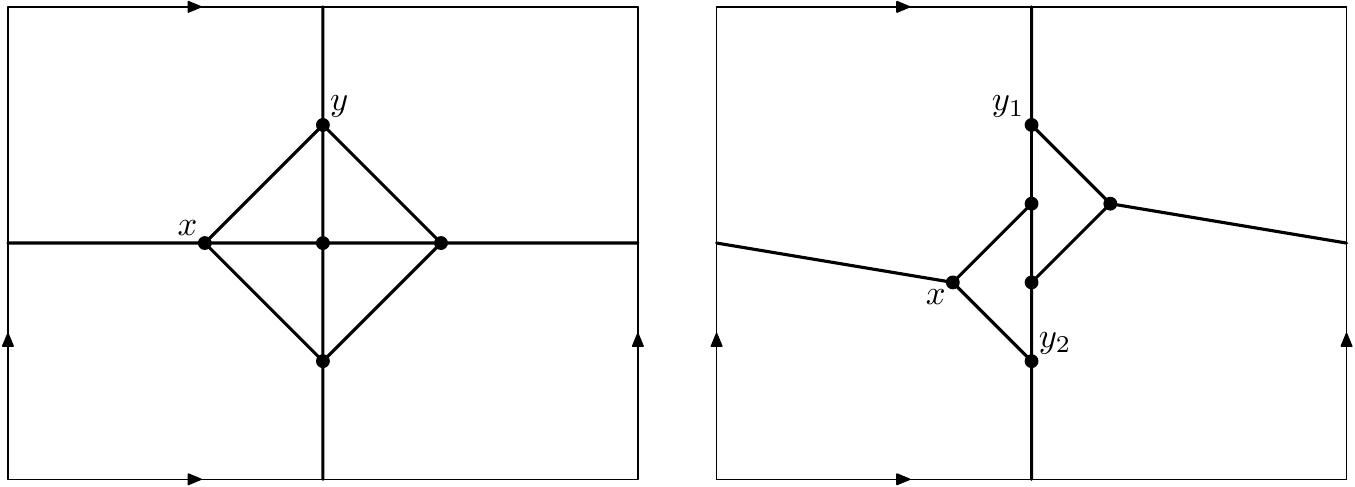}
  \caption{Kuratowski graphs and their two-vertex alternating embeddings in the torus.}
  \label{fg-kuratowski-alt}
\end{figure}

Note that both $K_5$ and $K_{3,3}$ are $xy$-alternating on the torus for any pair of vertices $x$ and $y$ (see~Fig.~\ref{fg-kuratowski-alt}).

For a graph $G$ and a vertex $x$ of $G$, the graph $G'$ is obtained by \df{splitting} $G$ at $x$ if $x$ is replaced by two adjacent vertices $x_1$ and $x_2$
and edges incident with $x$ in $G$ are distributed arbitrarily to $x_1$ and $x_2$ in $G'$.
By doing the same except that $x_1$ and $x_2$ are non-adjacent, a resulting graph $G'$ is said to be obtained by \df{cutting} of $G$ at $x$.

Suppose that a graph $G$ is embedded in some surface $\SSS$.
Let $\gamma$ be a simple closed curve in $\SSS$ that intersects the embedded graph $G$ only at vertices of $G$.
The number of vertices in $\gamma \cap V(G)$ is called the \df{width} of $\gamma$ (with respect to the embedded graph).
If $\gamma$ intersects $G$ at a vertex $z$, then it separates the edges incident with $z$ into two parts, \df{$\gamma$-sides} at $z$,
according to their appearance in the local rotation around $z$.
The graph obtained by cutting $G$ at each vertex $v$ in $\gamma \cap V(G)$ using the $\gamma$-sides to partition the edges
is said to be obtained by \df{cutting $G$ along $\gamma$}.
The curve $\gamma$ also induces the \df{cutting} of the surface $\SSS$ along $\gamma$, and the cut graph is embedded in the cut surface.
A curve is \df{orientizing} for a $\Pi$-embedded graph $G$ if cutting $G$ along $\gamma$ yields an orientable embedding of the
resulting graph using the embedding induced by $\Pi$.
The \df{orientizing face-width} of $G$ is the minimum width of an orientizing curve.

The next lemma outlines three characterizations of $\A_{xy}^k$.

\begin{lemma}
\label{lm-alt-equiv}
  Let $G$ be a graph with terminals $x$ and $y$. 
  If $G$ does not embed into $\SSS_{k-1}$, then the following statements are equivalent:
  \begin{enumerate}[\rm(i)]
  \item 
    $G$ is in $\A_{xy}^k$.
  \item
    $G$ has an embedding $\Pi$ into $\NN_{2k-1}$ with an orientizing $1$-sided simple closed curve $\gamma$ of width 2 going through $x$ and $y$.
  \item
    $G$ can be cut at $x$ and $y$ so that the resulting graph embeds into $\SSS_{k-1}$ with $x_1$, $y_1$, $x_2$ and $y_2$ appearing on a common face (in the stated order).
  \item
    $G^*$ embeds into $\SSS_k$.
  \end{enumerate}
\end{lemma}

The proof of Lemma~\ref{lm-alt-equiv} uses the following result by Archdeacon and Huneke~\cite{archdeacon-1989}.

\begin{lemma}
\label{lm-alternating}
  Let $G$ be a $\Pi$-embedded graph and $W$ a $\Pi$-facial walk. 
  If two vertices $x$ and $y$ appear twice in $W$ in the alternating order $x, y, x, y$, 
  then there exists an embedding $\Pi'$ of $G$ of Euler genus $\eg(\Pi)-1$ such that
  every $\Pi$-facial walk is $\Pi'$-facial except for $W$ which turns into
  two $\Pi'$-facial walks $W_1$ and $W_2$, each of which contains both $x$ and $y$.
  Moreover, the curve $\gamma$ passing through $x$ and $y$ and the faces $W_1$ and $W_2$ is $1$-sided in $\Pi'$
  and the signatures of edges in $\Pi'$ differ from $\Pi$ only by switching the signatures of a $\gamma$-side at $x$
  and a $\gamma$-side at $y$.
\end{lemma}

\begin{proof}[Proof of Lemma~\ref{lm-alt-equiv}]
  The equivalence of (i) and (iv) was proven by Decker et al.~\cite{decker-1985}.

  (i)$\imp$(ii):
  Since $G$ does not embed into $\SSS_{k-1}$, it $\Pi$-embeds into $\SSS_k$ with $x$ and $y$ alternating in a $\Pi$-face $W$.
  By Lemma~\ref{lm-alternating}, there is an embedding $\Pi'$ of Euler genus $2k-1$ with two $\Pi'$-faces $W_1$ and $W_2$,
  both containing $x$ and $y$. The curve $\gamma$ obtained by connecting vertices $x$ and $y$ in both faces $W_1$ and $W_2$
  is the sought $1$-sided curve of width 2.
  Since the signatures of edges in $\Pi$ are positive, the edges of negative signature in $\Pi'$ form
  two $\gamma$-sides of $x$ and $y$ (respectively).
  Thus cutting $G$ along $\gamma$ yields an orientable embedding and $\gamma$ is orientizing.

  (ii)$\imp$(iii):
  Cutting along the 1-sided orientizing curve $\gamma$ yields an orientable embedding $\Pi$ of genus $k-1$.
  Since $\gamma$ is 1-sided, the vertices obtained by cutting $G$ along $y$ lie on a common face in the interlaced order.

  (iii)$\imp$(i):
  Take an embedding $\Pi$ of the resulting graph $G'$ into $\SSS_{k-1}$ with $x_1, y_1, x_2, y_2$ on a common face $W$.
  Let $G'' = G' + x_1x_2 + y_1y_2$.
  We extend $\Pi$ to an embedding $\Pi'$ of $G''$ into $\SSS_k$ by embedding the new edges into $W$ (and adding a handle).
  The number of faces of $\Pi'$ stays the same but the number of edges is increased by two.
  Thus $g(\Pi') = g(\Pi) + 1$.
  By contracting the edges $x_1x_2$ and $y_1y_2$, we obtain $G$ and its $xy$-alternating embedding in $\SSS_k$.
\end{proof}

The classical result of Robertson and Seymour~\cite{robertson-1990} asserts that the set of obstructions for each minor-closed family of graphs is finite.
In particular, this implies that $\Forb(\SSS_k)$ is finite for each $k \ge 0$. A \df{topological obstruction} $G$ for $\SSS_k$ is a graph with no vertices of degree 2 that does not embed in $\SSS_k$
but each proper subgraph of $G$ does.
Since minors and \df{topological minors} ($H$ is a \df{topological minor} of $G$ if $G$ contains a subdivision of $H$ as a subgraph) 
are closely related, the set $\Forb^*(\SSS_k)$ of topological obstructions is also finite for each $k \ge 0$ (see~\cite[Prop.~6.1.1.]{mohar-book}).
Unfortunately, since the graphs in the classes $\A_{xy}^k$ have terminals, the result of Robertson and Seymour does not directly apply and thus it is not clear {\it a priori\/} whether the sets $\F_{xy}^k$ are finite.
The next lemma shows that the graphs in $\F_{xy}^k$ are derived from graphs in $\Forb^*(\SSS_k)$ and thus the finiteness of $\Forb^*(\SSS_k)$ implies the finiteness of $\F_{xy}^k$.

\begin{lemma}
\label{lm-finiteness}
  Let $G \in \F_{xy}^k$. Then precisely one of the graphs $G$, $G^+$, or $G^*$ belongs to $\Forb^*(\SSS_k)$.
\end{lemma}

\begin{proof}
  For each $e \in E(G)$ (possibly $e = xy$), we have $G - e \in \A_{xy}^k$.
  Therefore, we have that $g(G-e) \le k$ and, if $g(G - e) = k$, then $G - e$ is $xy$-alternating on $\SSS_k$.
  If $g(G) > k$, then $G \in \Forb^*(\SSS_k)$, since $g(G - e) \le k$ for each $e \in E(G)$.
  Thus we may assume that $g(G) = k$ and $G$ is {\em not} $xy$-alternating on $\SSS_k$.

  Suppose that $g(G^+) > g(G)$. For each $e \in E(G)$, we have $G - e \in \A_{xy}^k$, hence either $g(G - e) = k-1$ and thus $g(G^+ - e) = k$, or $G - e$ is $xy$-alternating
  on $\SSS_k$ and then $g(G^+ - e) = k$ since the edge $xy$ can be embedded into the $xy$-alternating face. Therefore, $G^+ \in \Forb^*(\SSS_k)$.

  Suppose now that $g(G^+) = g(G)$. We shall show that $G^* \in \Forb^*(\SSS_k)$.
  Since $G$ is not $xy$-alternating on $\SSS_k$, $g(G^*) > g(G)$ by Lemma~\ref{lm-alt-equiv}.
  For $e \in E(G)$, either $g(G -e) = k-1$ and thus $g(G^* - e) \le k$ by~(\ref{eq-ori}), 
  or $G -e$ is $xy$-alternating on $\SSS_k$ and so $g(G^* - e) = k$ also by~(\ref{eq-ori}).
  Let $H$ be the $xy$-bridge of $G^*$ induced by the edges not in $G$, which is isomorphic to $K_5$ minus an edge.
  For $e \in E(H)$, since $H^+ - e$ is planar, Lemma~\ref{lm-planar-patch} gives that $g(G^* - e) = g(G^+) = g(G) = k$.
  This shows that $G^* \in \Forb^*(\SSS_k)$.

  In conclusion, $G$, $G^+$, or $G^*$ belongs to $\Forb^*(\SSS_k)$, and it is clear that only one of these graphs is in $\Forb^*(\SSS_k)$
  since they are topological minors of each other.
\end{proof}

Lemma~\ref{lm-finiteness} has the following immediate corollary.

\begin{corollary}
\label{cr-finiteness}
  For $k \ge 1$, the class of graphs $\F_{xy}^k$ is finite.
\end{corollary}

\begin{proof}
Let $\F$ be the family of all graphs with two terminals obtained from graphs $H \in \Forb^*(\SSS_k)$
by declaring two vertices of $H$ to be terminals (in all possible ways), by declaring two adjacent vertices
to be terminals and deleting the edge joining them or by removing a bridge isomorphic to $K_5$ minus an edge
and declaring its two vertices of attachments to be terminals.
By Lemma~\ref{lm-finiteness}, $\F_{xy}^k \sss \F$. This completes the proof since $\Forb^*(\SSS_k)$ is finite.
\end{proof}

\begin{lemma}
  \label{lm-alt-jump}
  For $k \ge 1$, let $G \in \F_{xy}^k$.
  If  $G$ is not embeddable into $\SSS_k$,
  then $xy \not\in E(G)$, $G$ is $xy$-alternating on $\SSS_{k+1}$ and an obstruction for $\SSS_k$.
\end{lemma}

\begin{proof}
  If $xy \in E(G)$, then $G - xy \in \A_{xy}^k$. Since $g(G) > k$, we have $g(G - xy) = k$ and
  thus $G - xy$ has an $xy$-alternating embedding in $\SSS_k$.
  But then $G$ also embeds in $\SSS_k$. This contradiction shows that $xy \not\in E(G)$.

  Since $G\minor e \in \A_{xy}^k$ for every edge $e \in E(G)$ and each minor operation $\minor$,
  $G\minor e$ embeds into $\SSS_k$. Hence $G$ is an obstruction for $\SSS_k$.

  Let us construct an $xy$-alternating embedding in $\SSS_{k+1}$.
  Let $e = uv$ be an arbitrary edge in $G$ and consider the graph $G - e$.
  Clearly, the genus of $G - e$ cannot drop by more than one and since $G \in \F_{xy}^k$, there has to
  be an $xy$-alternating embedding of $G - e$ in $\SSS_k$.
  Let $\Pi$ be this $xy$-alternating embedding of $G - e$ in $\SSS_k$ and
  let $W$ be an $xy$-alternating $\Pi$-face.
  Pick two arbitrary $\Pi$-faces $W_u$ and $W_v$ incident with $u$ and $v$, respectively.
  Since $u$ and $v$ are not $\Pi$-cofacial, $W_u$ and $W_v$ are distinct.
  Extend $\Pi$ to an embedding $\Pi'$ of $G$ in $\SSS_{k+1}$ by adding a handle into faces $W_u$ and $W_v$.
  Since at most one of $W_u$ or $W_v$ is $W$, 
  the $xy$-alternating $\Pi$-face $W$ is extended to an $xy$-alternating $\Pi'$-face.
\end{proof}

The following is an immediate corollary of Lemma~\ref{lm-alt-jump}.

\begin{corollary}
  For $k \ge 1$, we have $\F_{xy}^k \sss \A_{xy}^{k+1}$.
\end{corollary}

We think that the scenario forced by Lemma~\ref{lm-alt-jump}, when $G \in \F_{xy}^k$ is not embeddable in $\SSS_k$, is quite unlikely, and
we would like to pose the following conjecture.

\begin{conjecture}
  Let $G$ be in $\F_{xy}^k$. Then $G$ embeds in $\SSS_k$.
\end{conjecture}

In this paper we confirm the conjecture for $k = 1$.

\section{Basic classification}
\label{sc-basic}

To classify all minimal obstructions for the torus of connectivity 2, we aim to understand the class $\A_{xy}^1$ of $xy$-alternating graphs on the torus
and the set  $\F_{xy}^1$ of its obstructions.

Lemma~\ref{lm-alt-equiv} gives the following characterizations of $\A_{xy}^1$.

\begin{corollary}
\label{cr-alt-equiv}
  Let $G$ be a non-planar graph with terminals $x$ and $y$. The following statements are equivalent:
  \begin{enumerate}[\rm(i)]
  \item 
    $G$ is in $\A_{xy}^1$.
  \item
    $G$ has an embedding $\Pi$ into the projective plane of face-width 2 with a non-contractible curve of width 2 going through $x$ and $y$.
  \item
    $G$ can be cut at $x$ and $y$ so that the resulting graph is planar with $x_1$, $x_2$, $y_1$ and $y_2$ on a common face.
  \item
    $G^*$ embeds into the torus.
  \end{enumerate}
\end{corollary}

By Corollary~\ref{cr-alt-equiv}, a non-planar graph $G$ belongs to $\A_{xy}^1$ if and only if the vertices $x$ and $y$ can be split so that
the resulting graph is planar with the new vertices on a common face. This implies that $G/xy$ is planar.

\begin{corollary}
\label{cr-planar-contraction}
  If $G$ is a non-planar graph in $\A_{xy}^1$, then $G /xy$ is planar.
\end{corollary}

We will show below that, if $G/xy$ is non-planar, then there is a Kuratowski subgraph in $G$
with a K-graph disjoint from $x$ and $y$.
%
%
%
The following lemma by Juvan et al.~\cite{juvan-1997} allows us to choose a subgraph without local bridges provided that we have an almost 3-connected graph.
Let $K$ be a subgraph of $G$. The graph $G$ is \df{$3$-connected modulo $K$} if for every vertex set $U \sss V(G)$ with at most 2 elements, 
every connected component of $G - U$ contains a branch vertex of $K$.

\begin{lemma}[Juvan, Marin\v{c}ek and Mohar~\cite{juvan-1997}]
\label{lm-no-local-bridges}
Let $K$ be a subgraph of a graph $G$. 
If $G$ is $3$-connected modulo $K$, then $G$ contains a subgraph $K'$ such that
\begin{enumerate}[\rm(a)]
\item 
  $K'$ is homeomorphic to $K$ and has the same branch vertices as $K$.
\item
  For each branch $e$ of $K$, the corresponding branch $e'$ of $K'$ joins
  the same pair of branch vertices as $e$ and is contained in the union of $e$ and all $K$-bridges
  that are local on $e$.
\item
 $K'$ has no local bridges.   
\end{enumerate}
\end{lemma}

Now, we are ready to prove that if $G/xy$ is non-planar, then there is a Kuratowski subgraph in $G$
with a K-graph disjoint from $x$ and $y$.

\begin{lemma}
  \label{lm-kuratowski-pinch}
  Let $G$ be a non-planar graph and $x, y \in V(G)$.
  If $G /xy$ is non-planar, then $G$ contains a K-graph disjoint from $x$ and $y$.
\end{lemma}

\begin{proof}
  Suppose that the conclusion of the lemma is false.
  Let $G$ be a counterexample with $|V(G)| + |E(G)|$ minimum.
  It is easy to see that $G$ is connected.
  If $G - x$ is non-planar, then by Theorem~\ref{th-kuratowski}, $G - x$ contains a Kuratowski graph $K$ 
  and thus $K - y$ contains a K-graph in $G$ that is disjoint from $x$ and $y$.
  Hence $G - x$ is planar. Similarly, $G - y$ is planar.

  Let $K$ be a Kuratowski subgraph in $G /xy$ and $L$ a K-graph contained in $K - v_{xy}$.
  Let $B_x$ and $B_y$ be the $L$-bridges of $G$ containing $x$ and $y$, respectively.
  Necessarily, $B_x$ and $B_y$ are different $L$-bridges of $G$ since otherwise $L$ is a K-graph in $G$ disjoint from $x$ and $y$.
  We aim to get rid of the local $L$-bridges by applying Lemma~\ref{lm-no-local-bridges}
  but also preserve the property that the graph is a K-graph in $G/xy$ that is disjoint from $x$ and $y$.
  In order to achieve that, we consider the graph $\hat{G} = G - B_x^\circ - B_y^\circ - w_1w_2$ in the case when $L$ is isomorphic to $K_{2,3}$,
  $w_1, w_2$ are the vertices of degree 3 in $L$, and $w_1w_2 \in E(G)$.
  Otherwise, let $\hat{G} = G - B_x^\circ - B_y^\circ$.

  If $\hat{G}$ is not 3-connected modulo $L$, then there is a (minimal) vertex set $U$ with $|U| \le 2$ such that a $U$-bridge $C$  does not
  contain any branch vertex (in $C^\circ$). 
  If $|U| \le 1$, then $C$ is  a block of $\hat{G}$.
  Since genus is additive over blocks (see~\cite{battle-1962}), the block $C$ is planar and its removal from $G$ yields a subgraph of $G$ that satisfies
  the assumptions of the lemma. 
  This is a contradiction with the choice of $G$ being minimal.
  Thus $U$ contains exactly two vertices, $u$ and $v$, and there is a path in $C$
  that connects $u$ and $v$. Let $G'$ be the graph obtained from $G$ by contracting $C$ into a single edge $uv$.
  Since $C$ does not contain $x$ and $y$, if $C + uv$ is non-planar, then $C$ contains a K-graph disjoint from $x$ and $y$ in $G$.
  Hence $C + uv$ is planar and Lemma~\ref{lm-planar-patch} gives that $G'$ is non-planar.
  It is not difficult to see that $G' /xy$ is also non-planar. 
  By the choice of $G$, there is a K-graph $L'$ in $G'$ disjoint from $x$ and $y$.
  Since the edge $uv$ in $G'$ can be replaced in $G$ by a path in $C$,
  $L'$ induces in a straightforward way a K-graph in $G$ disjoint from $x$ and $y$.

  Therefore, we may assume that $\hat{G}$ is 3-connected modulo $L$.
  By Lemma~\ref{lm-no-local-bridges}, there exists  a subgraph $L'$ of $\hat{G}$ homeomorphic to $L$
  that has no local bridges, and has the same branch vertices as $K'$ and also satisfies property (b) of Lemma~\ref{lm-no-local-bridges}.
  Note that, since $K_{2,3}$ and $K_4$ are uniquely embeddable in the plane, $L'$ has a unique planar embedding $\Pi$.
  Let $B_x'$ and $B_y'$ be the $L'$-bridges in $G$ containing $x$ and $y$, respectively.
  By using (b) of Lemma~\ref{lm-no-local-bridges}, it is not difficult to check that $L'$ is still a K-graph in $G/xy$.
  It follows that $B_x'$ and $B_y'$ are different $L'$-bridges in $G$.

\begin{figure}
  \centering
  \includegraphics{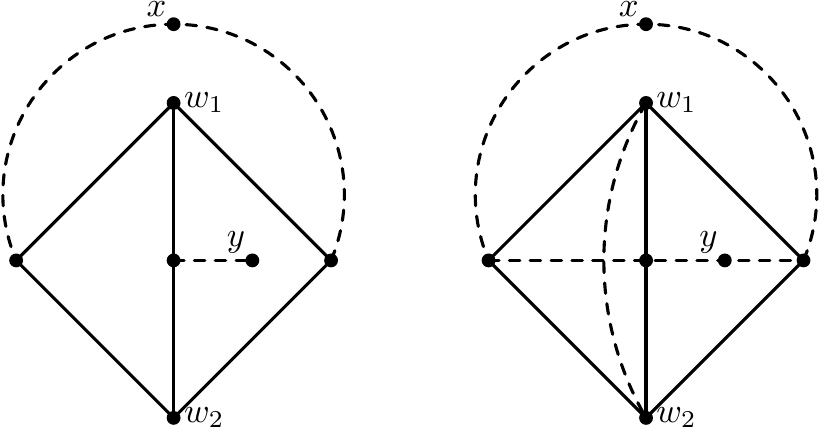}
  \caption{A case in the proof of Lemma~\ref{lm-kuratowski-pinch}.}
  \label{fg-proof-case}
\end{figure}

  \begin{casesblock}
    \case{$L'$ is a subdivision of $K_4$ or $w_1w_2 \not\in E(G)$}
    
    Since $G - B_x'^\circ$ and $G - B_y'^\circ$ are planar, each $L'$-bridge can be embedded into some $\Pi$-face.
    Since only $B_x'$ and $B_y'$ can be local $L'$-bridges in $G$, each other $L'$-bridge in $G$ embeds into a unique $\Pi$-face.
    Since the vertices of the union of the attachments of $B_x'$ and $B_y'$ do not lie on a single $\Pi$-face, the bridges $B_x'$ and $B_y'$ embed into different $\Pi$-faces.
    We conclude that each $L'$-bridge in $G$ can be assigned a $\Pi$-face such that
    all bridges assigned to a single $\Pi$-face can be embedded there simultaneously.
    Hence $G$ is planar --- a contradiction.
    
    \case{$L$ is a subdivision of $K_{2,3}$ and $w_1w_2 \in E(G)$}
    Consider the graph $G' = G - w_1w_2$.
    Since $G'$ is a subgraph of $G$ and  $G'/xy$ is non-planar, $G'$ is planar by the choice of $G$ .
    Since the planar embedding of $G'$ cannot be extended into a planar embedding of $G$ by adding the edge $w_1w_2$
    into one of the three $\Pi$-faces, there are three paths $P_1, P_2, P_3$ that connect the three pairs of open branches of $L'$, respectively (see Fig.~\ref{fg-proof-case}).
    Let $L''$ be the subgraph of $G$ that consists of $w_1w_2$, the path $P_i$ that is embedded in the $\Pi$-face containing neither $x$ nor $y$
    and the two branches of $L'$ that $P_i$ connects to.
    It is easy to see that $L''$ forms a K-graph in $G$ that is disjoint from $x$ and $y$, a contradiction.
  \end{casesblock}%
\end{proof}

Lemma~\ref{lm-kuratowski-pinch} leads to the following dichotomy of graphs is $\F_{xy}^1$.

\begin{lemma}
\label{lm-k-graph}
  Let $G$ be a graph in $\F_{xy}^1$. Then one of the following is true.
  \begin{enumerate}[\rm(i)]
  \item 
    $G$ is a split of a Kuratowski graph with $x$ and $y$ being the two vertices resulting after the split (see Fig.~\ref{fg-split})
    or $G$ is a Kuratowski graph plus one or two isolated vertices that are terminals.
  \item
    $G/xy$ is planar.
  \end{enumerate}
\end{lemma}

\begin{figure}
  \centering
  \includegraphics{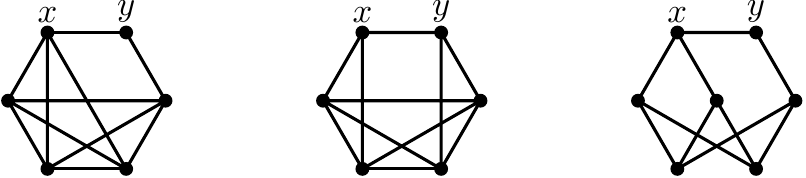}
  \caption{Splits of Kuratowski graphs.}
  \label{fg-split}
\end{figure}

\begin{proof}
  Suppose that $G$ does not satisfy (ii). 
  By Lemma~\ref{lm-kuratowski-pinch}, there is a Kuratowski subgraph $K$ in $G$ with a K-graph $L$ disjoint from $x$ and $y$.
  If there is an edge $e$ and a minor operation $\minor$ such that $G \minor e$ still contains a K-graph disjoint from $x$ and $y$,
  then $(G \minor e) /xy$ is non-planar and thus $G \minor e \not\in A_{xy}^1$ by Corollary~\ref{cr-planar-contraction}.
  Hence $E(G) = E(K)$. If $e$ is a subdivided edge of $K$, then $G / e$ still contains a K-graph disjoint from $x$ and $y$ 
  unless a terminal and a branch vertex of $K$ are the endvertices of $e$.
  Now it is easy to see that $G$ satisfies (i).
\end{proof}

\section{$XY$-labelled graphs}

Let $G$ be a graph with terminals $x$ and $y$. 
To investigate graphs in $G \in \F_{xy}^1$ where $G /xy$ is planar, we study the graph $H = G - x - y$.
Let us label each vertex of $H$ by the label $X$ ($Y$) if it is adjacent to $x$ ($y$) in $G$.
Thus each vertex of $H$ is given up to two labels.
Let $\lab(v)$ denote the set of labels given to the vertex $v$ of $H$.
A vertex $v$ is \df{labelled} if $\lab(v)$ is non-empty. 
The graph $H$ together with the labels carries all information about $G$. 
Let us call $H$ an \df{$XY$-labelled graph}.
The notion of a minor of a graph is extended to $XY$-labelled graphs naturally: 
an $XY$-labelled graph $H_1$ is a minor of an $XY$-labelled graph $H_2$
if the graph with terminals corresponding to $H_1$ is a minor of the graph with terminals corresponding to $H_2$.
For example, the deletion of a label is a minor operation that corresponds to an edge deletion and, 
when contracting an edge $uv$ in an $XY$-labelled graph, 
the resulting vertex is labelled by $\lab(u) \cup \lab(v)$.

Another useful representation of $G$ is as follows.
Consider the multigraph $\hat{H}$ and  the vertex $v_{xy}$ obtained by identification of $x$ and $y$ in $G$ (in contrast to the simple graph $G /xy$ used in the previous sections).
Label each edge $e$ of $\hat{H}$ incident to $v_{xy}$ by the label $X$ ($Y$) if the edge was incident to $x$ ($y$) in $G$.
Let $\Pi$ be a planar embedding of $\hat{H}$. 
The local rotation around $v_{xy}$ gives a cyclic sequence $S$ of labels that appear on the edges incident with $v_{xy}$.
Call $S$ a \df{label sequence} of $\hat{H}$.
A \df{label transition} in a label sequence is a pair of (cyclically) consecutive labels that are different.
The \df{number of transitions} $\trans(Q)$ of $S$ is the number of label transitions in $S$.
In the case when $S$ contains only two different labels, $\trans(Q)$ is a multiple of 2.
Thus we say that a label sequence $S$ is \df{$k$-alternating} if $\trans(Q) = 2k$.
A planar embedding of $\hat{H}$ is \df{$k$-alternating} if the induced label sequence is $k$-alternating
and
$H$ is called \df{$k$-alternating} if $\hat{H}$ admits a $k$-alternating embedding in the plane.
Note that Lemma~\ref{cr-alt-equiv} implies that, if $H$ is 2-alternating, then the corresponding graph $G$ is in $\A_{xy}^1$.


When $H$ is connected, a planar embedding of $\hat{H}$ induces a planar embedding of $H$
with a special face $W$ in which $v_{xy}$ is embedded. 
Call the cyclic sequence of vertices of $W$ (with some possibly appearing more than once) a \df{boundary} of $H$.
If $H$ is 2-connected, then $W$ is a  cycle of $H$ (see~\cite[Thm.~2.2.3]{mohar-book}).
To understand when a planar embedding of $\hat{H}$ induces a 2-alternating label sequence,
we study the possible boundaries of $H$. 
If $M$ is a block of $H$ that is not an edge, then a boundary of $H$ induces a \df{boundary cycle} in $M$.

A sequence $R = v_1, \ldots, v_k$ of consecutive vertices on a boundary $Q$ is called an \df{$X$-block} in $Q$
if no vertices in $R$ except possibly the endvertices $v_1$ and $v_k$ are labelled with $Y$.
Define a $Y$-block similarly. 

The following lemma states the observation that, if two $X$-blocks contain all vertices that are labelled $X$,
then it is easy to construct a 2-alternating embedding of $\hat{H}$.
In this case, we say that the labels $X$ are \df{covered} by the two $X$-blocks.

\begin{lemma}
\label{lm-2-blocks}
  Let $H$ be an $XY$-labelled graph, $Q$ a boundary of $H$, and $A \in \{X, Y\}$.
  If the $A$-labelled vertices of $H$ are covered by two $A$-blocks in $Q$,
  then $H$ is 2-alternating.
\end{lemma}

For $A \in \{X, Y\}$, an induced subgraph $H'$ of $H$ \df{contains the label $A$} if there is a vertex in $H'$ labelled $A$.
Let $S = A_1\ldots A_k$ be a label sequence. 
Here we consider $S$ as a linear label sequence as opposed to cyclic.
Let $R$ be a subsequence of a boundary of $H$. We say that $R$ \df{contains the label sequence $S$} if there
are distinct vertices $v_1, \ldots, v_k$ that appear in $R$ in this order (or the reverse order)
and $v_i$ is labelled $A_i$ for $i=1, \ldots, k$.
We say that $H'$ \df{contains the label sequence $S$} if for every boundary $Q$ of $H$, the subsequence of $Q$ induced by $V(H')$
contains the label sequence $S$.
Let $B$ be a block of $H$ and $v$ a vertex of $B$. We say that label $A$ is \df{attached} to $B$ at $v$ if
either $v$ is labelled $A$ or there is a $v$-bridge in $H$ not containing $B$ that contains $A$.

\begin{lemma}
  \label{lm-alt-boundary}
  Let $H$ be an $XY$-labelled graph such that at most four vertices of $H$ have both labels $X$ and $Y$.
  If $H$ is not 2-alternating,
  then $H$ contains the label sequence $XYXYXY$.
\end{lemma}

\begin{proof}
  Suppose that $H$ is not 2-alternating and let $Q$ be a boundary of $H$.
  Let $R$ be a subsequence of $Q$ with no unlabelled vertices such that each labelled vertex appears in $R$ exactly once.
  A stronger claim is proved instead. If $R$ does not contain the label sequence $XYXYXY$,
  then the labels of vertices in $R$ can be arranged in the order given by $R$ to obtain a 2-alternating sequence of labels.
  Suppose that this is not true and choose a counter-example $R$ with minimum total number of labels.

  Suppose there are cyclically consecutive vertices $u$ and $v$ in $R$ such that both $u$  and $v$ have label $A$ and $v$ has only one label.
  By deleting $A$ from $u$ we obtain a sequence $R'$ with smaller total number of labels.
  By the construction of $R'$, $R'$ does not contain the label sequence $XYXYXY$.
  Thus there is a 2-alternating label sequence $S'$ of labels in $R'$. By inserting the label $A$ before the occurence of $A$ at $v$,
  we obtain a valid 2-alternating label sequence for $R$.
  Therefore, every two consecutive vertices in $R$ have either distinct labels or both labels.

  Two cases remain: Either $R$ contains at most four labelled vertices, all with both labels, or
  there are at most four vertices that have alternating labels 
  (six vertices give the label sequence $XYXYXY$ and five vertices are not possible because of parity).
  In both cases, we see immediately that the labels in $R$ can be arranged into a 2-alternating label sequence.
\end{proof}

For graphs in $\F_{xy}^1$, Lemma~\ref{lm-planar-patch} gives the following result.

\begin{corollary}
  \label{cr-planar-patch}
  Let $G \in \F_{xy}^1$ and let $\{u,v\}$ be a 2-vertex-cut in $G$.
  If $C$ is a non-trivial $uv$-bridge such that $C + uv$ is planar, then
  $C - u - v$ contains a terminal.
\end{corollary}

The following lemma describes the structure of a graphs in $\F_{xy}^1$ when the $XY$-labelled graph is disconnected.

\begin{figure}
  \centering
  \includegraphics{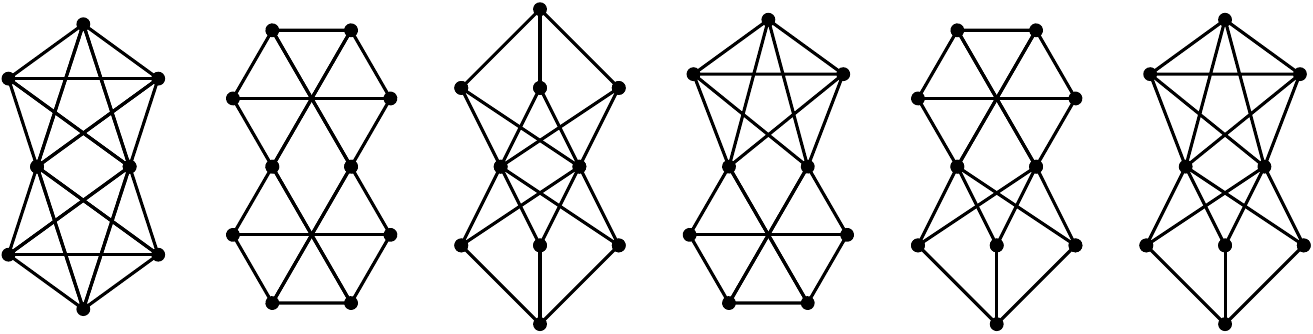}
  \caption{The two-sums of Kuratowski graphs.}
  \label{fg-kuratowski-sum}
\end{figure}

\begin{lemma}
\label{lm-alt-disconnected}
  Let $G$ be a graph in $\F_{xy}^1$ such that $G /xy$ is planar and let $H$ be the $XY$-labelled graph corresponding to $G$.
  If $H$ is disconnected, then 
  $G$ is an $xy$-sum of two Kuratowski graphs and $xy \not\in E(G)$ (this yields precisely six non-isomorphic graphs; see Fig.~\ref{fg-kuratowski-sum}).
\end{lemma}

\begin{figure}
  \centering
  \includegraphics{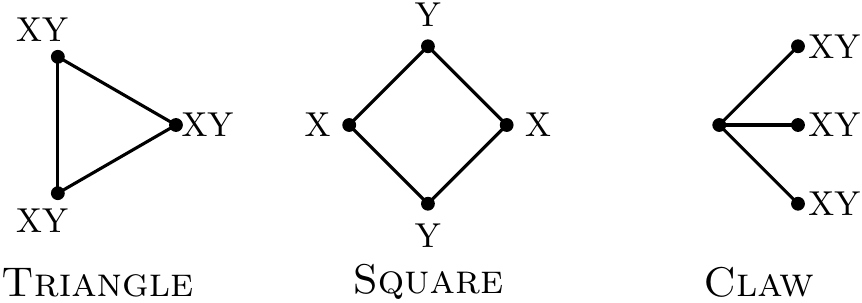}
  \caption{Kuratowski graphs as an alternating extension to outerplanar graphs.}
  \label{fg-alt-kuratowski}
\end{figure}

\begin{proof}
  Each $xy$-sum of two Kuratowski graphs (without the edge $xy$ even if it is present in a summand) is 
  a projective planar obstruction (see~\cite{archdeacon-1981}) and it is straightforward to check that
  if belongs to $\F_{xy}^1$.
  Fig.~\ref{fg-alt-kuratowski} shows the three possible $XY$-labelled blocks that arise.

  Since $H$ is disconnected, $G$ has at least two non-trivial $xy$-bridges $C_1$ and $C_2$.
  Since neither $C_1 - x - y$ nor $C_2 - x - y$ contains a terminal, Corollary~\ref{cr-planar-patch} gives that both $C_1 + xy$ and $C_2 + xy$ are non-planar.
  Hence $G$ contains an $xy$-sum of two Kuratowski graphs as a minor.
\end{proof}

\section{Connectivity 2}

This section is devoted to the proof of the following lemma characterizing graphs in $\F_{xy}^1$ that correspond to a 2-connected $XY$-labelled graph.

\begin{figure}
  \centering
  \includegraphics{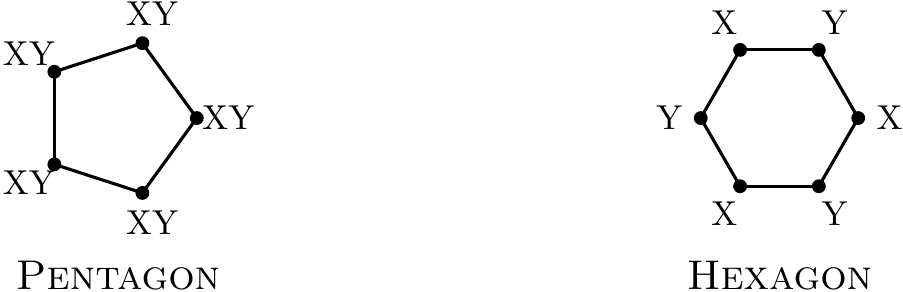}
  \caption{The 2-connected $XY$-labelled graphs that correspond to graphs in $\F_{xy}^1$ that contain the edge $xy$.}
  \label{fg-alt-xy}
\end{figure}

\begin{figure}
  \centering
  \includegraphics{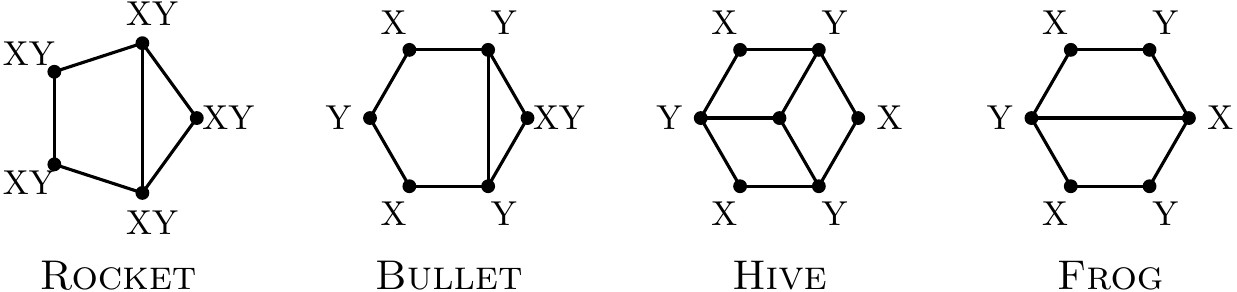}
  \caption{The 2-connected $XY$-labelled graphs that correspond to graphs in $\F_{xy}^1$ without the edge $xy$.}
  \label{fg-alt-2-con}
\end{figure}

\begin{lemma}
\label{lm-alt-2-con}
  Let $G$ be a graph in $\F_{xy}^1$ such that $G /xy$ is planar and such that  the $XY$-labelled graph $H$ corresponding to $G$ is 2-connected.
  If $xy\in E(G)$, then 
  $H$ is one of the graphs in Fig.~\ref{fg-alt-xy}.
  Otherwise, $H$ is one of the graphs in Fig.~\ref{fg-alt-2-con}.
\end{lemma}

First, we derive two lemmas that will be used in the proof of Lemma~\ref{lm-alt-2-con}.

\begin{lemma}
  \label{lm-common-bridge}
  Let $G$ be a graph that consists of a cycle $C$ and $C$-bridges $B_1, B_2$ such that
  all other $C$-bridges avoid each other.
  If $G$ is non-planar, then there is a $C$-bridge $B$ (different from $B_1$ and $B_2$)
  such that $B$, $B_1$, and $B_2$ all pairwise overlap.
\end{lemma}

\begin{proof}
  Let $\B$ be the set of $C$-bridges in $H$ different from $B_1$ and $B_2$.
  Since the bridges in $\B$ avoid each other, $\B$ forms an independent set in the bridge-overlap graph $H$ of $\B \cup \{B_1, B_2\}$.
  Since $G$ is non-planar, Theorem~\ref{th-bridge-overlap} asserts that $H$ is non-bipartite and thus contains an odd cycle.
  Since every edge in $H$ is incident with $B_1$ or $B_2$, this odd cycle is a triangle that consists of $B_1$, $B_2$ and a bridge $B \in \B$.
\end{proof}

\begin{lemma}
  \label{lm-segment}
  Let $H$ be an $XY$-labelled planar graph that consists of an $XY$-labelled cycle $C$
  and a $C$-bridge $B$. Let $C[w_1,w_2]$ be a segment of $C$ that contains
  all attachments of $B$.
  If $C$ contains all labels of $H$ and the graph with terminals corresponding to $H$ is non-planar,
  then $C(w_1, w_2)$ contains both labels.
\end{lemma}

\begin{proof}
  Let $G$ be the graph with terminals $x, y$ corresponding to $H$.
  Let $B_x$ and $B_y$ be the $C$-bridges that contain $x$ and $y$, respectively.
  Since $C$ contains all labels of $H$, $B_x$ and $B_y$ are stars attached only to $C$.
  By Lemma~\ref{lm-common-bridge}, the bridges $B$, $B_x$ and $B_y$ pairwise overlap.
  Theorem~\ref{th-disk-ext} implies that, for each $z \in \{x,y\}$, either
  \begin{enumerate}[\rm(i)]
  \item 
    there are disjoint crossing paths $P_1$ in $B$ and $P_2$ in $B_z$, or
  \item
    the bridges $B$ and $B_z$ have three vertices of attachment in common.
  \end{enumerate}
  Let $Z$ be the label corresponding to the vertex $z$.
  When (i) holds, $C(w_1, w_2)$ contains one of the endvertices of $P_2$ and thus contains the label $Z$.
  When (ii) holds, each attachment of $B$ is labelled $Z$. Since $C(w_1, w_2)$ contains at least one of the attachments of $B$,
  $C(w_1, w_2)$ contains the label $Z$.
  Therefore, $C(w_1, w_2)$ contains both labels $X$ and $Y$ as claimed.
\end{proof}

\begin{proof}[Proof of Lemma~\ref{lm-alt-2-con}]
  Let $C$ be a boundary cycle of $H$
  and $\Pi$ the corresponding planar embedding of $H$.

  Suppose that the edge $xy$ is present in $G$.
  By Lemma~\ref{lm-alt-boundary}, either $C$ contains the label sequence $XYXYXY$,
  and then $H$ has \gph{Hexagon} as a minor, or there are five vertices in $C$ with both labels, and
  then $H$ has \gph{Pentagon} as a minor.

  Therefore, we may assume that the edge $xy$ is not present in $G$.
  Let us consider the $C$-bridges $B_x$ and $B_y$ in $G$ that are the stars with centers $x$ and $y$, respectively.
  We may assume that, in $\Pi$, $C$ is the boundary of the infinite face.
  By Lemma~\ref{lm-common-bridge}, there is a $C$-bridge $B$ such that $B$, $B_x$, and $B_y$ pairwise overlap.

  Let us first consider the case when $C$ does not contain the label sequence $XYXYXY$.
  By Lemma~\ref{lm-alt-boundary}, $C$ contains five vertices with both labels. 
  Let $v_1, \ldots, v_5$ be the vertices with both labels.
  We may assume by symmetry that an attachment of $B$ lies in $C(v_1, v_3)$.
  If there is an attachment of $B$ in the segment $C(v_3, v_1)$, then $H$ has \gph{Rocket} as a minor.
  Otherwise, all attachments of $B$ are in the segment $C[v_1, v_3]$.
  Let $S$ be the support of $B$ in $C[v_1, v_3]$.
  By Lemma~\ref{lm-segment}, the segment $S$ (excluding the endvertices of $S$) contains both labels.
  Thus $H$ has \gph{Rocket} as a minor.

  Now, assume that $C$ contains the label sequence $XYXYXY$ and let $v_1, \ldots, v_6$ be the vertices manifesting that (so $X \in \lab(v_1)$, $Y \in \lab(v_2)$, etc.).
  Let $w_1,\ldots, w_k$ be the attachments of $B$. Note that $k \ge 2$.
  By symmetry, we may assume that $w_1$ lies in the segment $C(v_1, v_3)$.

  If all attachments of $B$ lie in $C[v_1, v_3]$,
  then the support $S$  of $B$ in $C[v_1, v_3]$ (excluding the endvertices of $S$) contains both labels by Lemma~\ref{lm-segment}.
  Thus $H$ has \gph{Bullet} as a minor.
  Hence we may assume that not all attachments of $B$ are in $C[v_1, v_3]$ and similarly in $C[v_2, v_4]$ and so on.
  If there is an attachment of $B$ in the segment $C(v_4, v_6)$, then $H$ has \gph{Frog} as a minor.
  Hence we may assume that all attachments of $B$ lie in the segment $C[v_6, v_4]$.

  By using reflection symmetry exchanging $v_1, v_3$ and $v_4,v_6$, since not all attachments of $B$ are in $C[v_1, v_3]$,
  there is an attachment $w_2$ of $B$ in the segment $C(v_3, v_4]$.
  By the same argument as above, there is no attachment of $B$ in $C(v_6, v_2)$.
  Since not all attachments of $B$ are in $C[v_2, v_4]$, the vertex $v_6$ is an attachment of $B$.
  We conclude that $H$ has \gph{Hive} as a minor.
\end{proof}

\section{Connectivity 1}

In this section, we describe all obstructions in $\F_{xy}^1$ that correspond to an $XY$-labelled graph of connectivity 1.

\begin{lemma}
\label{lm-alt-1-con}
  Let $G$ be a graph in $\F_{xy}^1$ such that $G /xy$ is planar and let $H$ be the $XY$-labelled graph corresponding to $G$.
  If $H$ has connectivity 1, then $H$ is one of the graphs in Fig.~\ref{fg-alt-1-con}.
  Furthermore, $xy \not\in E(G)$.
\end{lemma}

The following observation is useful.

\begin{lemma}
  \label{lm-triangle}
  Let $G$ be a graph and $uvw$ be a triangle in $G$.
  If $u$ has degree 3 in $G$,
  then every embedding of $G - vw$ into a surface can be extended into an embedding of $G$
  into the same surface.
\end{lemma}

\begin{proof}
  Let $H$ be the graph obtained from $G - vw$ by subdividing the edge incident to $u$ that is not in the triangle $uvw$.
  Then $G$ is the graph obtained from $H$ by applying a $\Delta$-operation on $u$.
  The result follows.
\end{proof}

For graphs in $\F_{xy}^1$, Lemma~\ref{lm-triangle} has the following consequence.

\begin{corollary}
  \label{cr-triangle}
  Let $G \in \F_{xy}^1$ and $uvw$ be a triangle in $G$.
  If $u$ has degree at most 3 in $G$, then $u$ is a terminal.
\end{corollary}

\begin{proof}
  Since $G - vw \in \A_{xy}^1$, either $G -vw$ is planar or $G - vw$ is $xy$-alternating on $\SSS_1$.
  By Lemma~\ref{lm-triangle}, the first outcome is not possible since then $G$ would be planar.
  In the second case, Lemma~\ref{lm-triangle} shows that the $xy$-alternating embedding of $G - vw$
  can be extended into an embedding of $G$ in $\SSS_1$ by embedding $vw$ along
  the path $vuw$. This extension would be $xy$-alternating if $u \not\in \{x, y\}$.
  Thus, $u$ is one of the terminals.
\end{proof}

The next lemma will be used throughout the rest of the paper.

\begin{lemma}
  \label{lm-two-blocks}
  Let $H$ be an $XY$-labelled graph that has distinct blocks $B_1$ and $B_2$.
  Suppose that each of $B_1$ and $B_2$ contains both labels $X$ and $Y$ on vertices that
  do not belong to another block.
  Let $G$ be the graph with terminals corresponding to $H$.
  If $H$ is not $1$-alternating, then $G$ is non-planar.
\end{lemma}

\begin{proof}
  Suppose for contradiction that $G$ is planar and take a planar embedding $\Pi$ of $G$.
  If $x$ and $y$ are cofacial in $\Pi$, then $\Pi$ gives a $1$-alternating embedding of $\hat{H}$.
  If $x$ and $y$ are not cofacial in $\Pi$, then there is a cycle $C$ in $H$ that separates $x$ and $y$ (since $x$ and $y$ lie inside different faces of the induced embedding of $H$).
  Since $C$ is a cycle of $H$, it intersects either $B_1$ or $B_2$ in at most one vertex.
  Say, $B_1$ shares at most one vertex with $C$ and is embedded on the other side of $C$ than $x$ is.
  By assumption, there is a vertex $v \in V(B_1) \sm V(C)$ that is labelled $X$.
  Clearly, $v$ and $x$ are not cofacial in $\Pi$ since they are separated by $C$.
  But $v$ and $x$ are adjacent and thus cofacial in $\Pi$, a contradiction.
\end{proof}

Let $C$ be a block in a graph $G$. The \df{$C$-bridge set} $B_v$ at a vertex $v$ of $C$
is the union of all $C$-bridges in $G$ that are attached to $v$.
The following lemma asserts several properties of $H$ and its labels and it is used to classify the graphs of connectivity $1$ in $\F_{xy}^1$.

\begin{lemma}
\label{lm-alt-struct}
  Let $G$ be a graph in $\F_{xy}^1$ such that $G /xy$ is planar and the corresponding $XY$-labelled graph $H$ has connectivity 1. 
  Then the following statements hold.
  \begin{enumerate}[\rm(S1)]
  \item\label{it-degree}
    Vertices of degree at most $2$ in $H$ are labelled.
    Leaves in $H$ have both labels.
  \item\label{it-endblock}
    If $B$ is an endblock of $H$, and $v$ is a cutvertex that separates $B$ from the rest of $H$,
    then the graph $B - v$ contains both labels.
  \item\label{it-cycle-inside}
    Let $M$ be a block of $H$ that is not an edge and $C$ a boundary cycle of $M$.
    Let $B$ be the subgraph of $M$ that consists of $C$-bridges in $M$.
    If $B$ is non-empty, then $H - B^\circ$ is not $2$-alternating.
  \item\label{it-structure}
    Each block of $H$ is either an edge or a cycle.
  \item\label{it-labels-cycle}
    Let $u$ be a vertex of degree 2 in $H$. If $u$ has only one label, then the neighbors of $u$ are not labelled by $\lab(u)$.
    In particular, if $P$ is a path in $H$ such that each vertex of $P$ has degree 2 in $H$, then
    either each vertex of $P$ has both labels or each vertex of $P$ has precisely one label
    that is different from the labels of its neighbors.
  \item\label{it-leaf}
    The neighbor of a leaf in $H$ is unlabelled.
  \item\label{it-single-leaf}
    Let $C$ be a cycle of $H$ and $T$ a $C$-bridge set that is a tree.
    If $H$ consists of at least three blocks, then $T$ contains at least two leaves of $H$.
  \item\label{it-triangle-att}
    Let $B$ be a block of $H$ that is a triangle and $v$ a vertex of $B$. 
    If $v$ is not a cutvertex, then it has both labels. 
    Otherwise, both labels are attached to $B$ at $v$.
  \end{enumerate}
\end{lemma}

\begin{proof} 
  Each property is proved separately.

  \point{(S\ref{it-degree}):}
  Vertex of degree 2 in $H$ with no label would be a vertex of degree 2 in $G$.
  Similarly, a vertex of degree 1 with at most one label would be a vertex of degree at most 2 in $G$.

  \point{(S\ref{it-endblock}):}
    Let $B$ be an endblock of $H$ and $v \in V(B)$ the cutvertex that separates $B$ from the rest of the graph.
    If $B$ is an edge, then the result follows from~(S\ref{it-degree}).
    Suppose for contradiction that $B - v$ does not contain the label $Y$.
    Since $G/xy$ is planar,  $B$ is either a planar block of $G$ or $B$ is in an $xv$-bridge $C$ of $G$
    such that $C + xv$ is planar.
    Corollary~\ref{cr-planar-patch} asserts that this cannot happen in $G$.

    \point{(S\ref{it-cycle-inside}):}
    Suppose $B$ is non-empty and $\Pi$ is a 2-alternating embedding of $\hat{H - B^\circ}$ in the plane.
    Suppose that there is an edge $e$ of $H - B^\circ$  with one end $v$ in $C$. 
    By construction of $H - B^\circ$, $e$ lies in a different $v$-block $B$ of $H$ than $C$.
    By~(S\ref{it-endblock}), there is a vertex $u$ in $B$  labelled $X$.
    Thus there is a path $P$ in $\hat{H - B^\circ}$ that connects $v_{xy}$ and $v$ and is internally disjoint from $C$.
    It follows that $e$ is embedded on the same side of $C$ in $\Pi$ as $x$ and $y$.
    We conclude that $C$ is a $\Pi$-face.
    By construction of $C$, $\Pi$ can be extended to a 2-alternating embedding of $\hat{H}$ by embedding $B$
    inside $C$ --- a contradiction.

    \point{(S\ref{it-structure}):}
    Let $M$ be a block of $H$ that is neither a cycle nor an edge.
    Let $C$ be a boundary cycle of $M$ and $B$ the subgraph that consists of $C$-bridges in $M$.
    By~(S\ref{it-cycle-inside}), $G - B^\circ$ is not $xy$-alternating on the torus.
    By~(S\ref{it-endblock}), $H - B^\circ$ contains two endblocks that contain both labels.
    By Lemma~\ref{lm-two-blocks}, $G - B^\circ$ is non-planar, a contradiction with $G - B^\circ \in \A_{xy}^1$.

    \point{(S\ref{it-labels-cycle}):}
    By~(S\ref{it-degree}), $u$ is labelled, say by $X$.
    If $v$ is a neighbor of $u$ with label $X$,
    then  $u$ is a vertex of degree 3 in the triangle $uvx$ which is not possible by Corollary~\ref{cr-triangle}
    unless $u$ is also labelled $Y$.

    \point{(S\ref{it-leaf}):}
    Let $v$ be a leaf and $u$ its neighbor.
    If $u$ is labelled, say with label $X$,
    then $v$ is a vertex of degree 3 in the triangle $vxu$ which is not possible by Corollary~\ref{cr-triangle}.

    \point{(S\ref{it-single-leaf}):}
    Let $C$ be a cycle and $T$ be a $C$-bridge set that is a tree.
    Assume that $H$ has at least 3 blocks and that $T$ contains only one leaf.
    We see that $T$ is a path and, by~(S\ref{it-leaf}) and~(S\ref{it-degree}), it is a path of length 1.
    Contract $T$ to $C$ to get $H'$. 
    Let $G'$ be the graph corresponding to $H'$.
    By the choice of $G$, $G'$ is either $xy$-alternating on the torus or planar.
    Since $H$ either contains 3 endblocks or two disjoint endblocks,
    if $G'$ is not $xy$-alternating on the torus, then Lemma~\ref{lm-two-blocks} gives that $G'$ is non-planar.
    Hence $G'$ is $xy$-alternating on the torus.
    Let $\Pi$ be a 2-alternating embedding of $\hat{H'}$ in the plane. 
    Uncontract $T$ to get a 2-alternating embedding of $\hat{H}$ --- a contradiction.

    \point{(S\ref{it-triangle-att}):}
    Let $v$ be a vertex in a triangle $C$ with at most one label.
    If $v$ is not a cutvertex, then  $v$ is has degree at most 3 in $G$.
    By Corollary~\ref{cr-triangle}, this is a contradiction.
    If $v$ is a cutvertex, then there is a $v$-bridge $B'$ that does not contain $C$.
    Since $B'$ contains an endblock of $H$, (S\ref{it-endblock}) implies that $B'$ contains both labels.
    These labels are attached to $B$ at $v$.
\end{proof}

\begin{figure}
  \centering
  \includegraphics{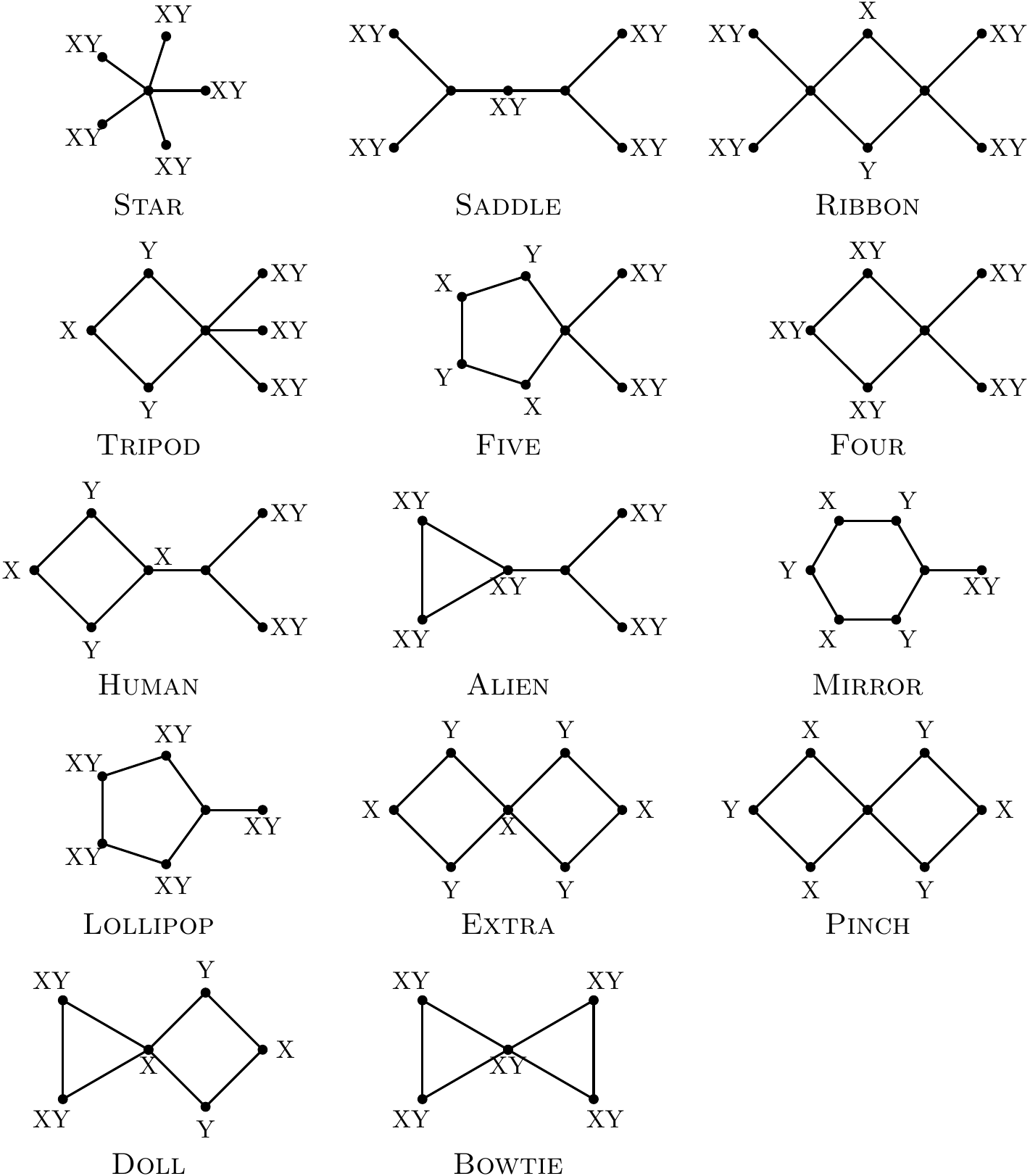}
  \caption{The $XY$-labelled graphs of connectivity 1 that correspond to graphs in $\F_{xy}^1$.}
  \label{fg-alt-1-con}
\end{figure}

We use the structural properties from Lemma~\ref{lm-alt-struct} to prove Lemma~\ref{lm-alt-1-con}.

\begin{proof}[Proof of Lemma~\ref{lm-alt-1-con}]
  Let $G$ and $H$ be as in the statement of the lemma.
  Our goal is to show that $H$
  has one of the graphs from Fig.~\ref{fg-alt-1-con} as a minor.

  If $H$ has at least five leaves, then 
  all leaves are labelled $X$ and $Y$, by~(S\ref{it-degree}). Since $H$ is connected, $H$ has \gph{Star} as a minor.
  We assume henceforth that $H$ has at most four leaves.

  By~(S\ref{it-structure}), every block of $H$ that is not an edge is a cycle.
  We split the discussion according to the number of cycles in $H$.

  \begin{casesblock}
  \case{$H$ is acyclic}
  Suppose  $H$ has $k$ leaves $w_1, \ldots, w_k$, where $k \le 4$. Let $u_1, \ldots, u_k$ be their neighbors (possibly not distinct).
  By~(S\ref{it-leaf}) and~(S\ref{it-degree}), vertices $u_i$ (where $i = 1,\ldots, k$) have no labels and are of degree at least 3.
  By a counting argument, there are at most two such vertices in $H$.
  If there is only one vertex $u$ of degree at least 3, $H$ is a star with center $u$ and 
  thus $H$ is a proper minor of \gph{Star} and hence $G$ is in $\A_{xy}^1$.
  Thus, there are two of them, say $u_1$ and $u_2$, and they are connected by a path $P$.
  If $P$ contains both labels $X$ and $Y$, then $H$ has \gph{Saddle} as a minor.
  If $P$ contains at most one of the labels, say $X$, then the two pairs of leaves are covered by two $Y$-blocks 
  and thus $G$ is in $\A_{xy}^1$ by Lemma~\ref{lm-2-blocks} --- a contradiction.
  
  \case{$H$ has precisely one cycle $C$}
  Since $C$ is the only cycle in $H$, every $C$-bridge is a tree attached to a vertex of $C$. 
  The proof is split according to the number of leaves of $H$.
  Note that $H$ has at least one leaf since $H$ is not 2-connected.

  \subcase{$H$ has precisely four leaves}
  If $C$ is an endblock, then a single $C$-bridge set $B_v$ contains all four leaves $w_1, \ldots, w_4$.
  By~(S\ref{it-endblock}),  $C - v$ contains both labels.
  Therefore, $H$ has \gph{Star} as a minor.

  Otherwise, by~(S\ref{it-single-leaf}), there are precisely two non-trivial $C$-bridge sets $B_{v_1}$ and $B_{v_2}$, and each contains two leaves. 
  Hence each of $B_{v_1} - v_1$ and $B_{v_2} - v_2$ contains at most one vertex of degree 3 in $H$.
  When $B_{v_1} - v_1$ contains a vertex of degree 3, let $u_1$ be this vertex.
  Otherwise, let $u_1 = v_1$. Define $u_2$ similarly.
  Note that $u_1$ and $u_2$ are unlabelled by~(S\ref{it-leaf}) and~(S\ref{it-degree}).
  If there is a path $P$ in $H$ connecting $u_1$ and $u_2$ and both labels $X$ and $Y$ appear on $P$,
  then $H$ has \gph{Saddle} as a minor.
  Let $P_1$ and $P_2$ be the two paths in $C$ connecting $v_1$ and $v_2$.
  If $P_1$ contains $X$  and $P_2$ contains $Y$ (or vice versa), then $H$ has \gph{Ribbon} (or \gph{Saddle}) as a minor.
  Otherwise, there is a label missing from $H - \{w_i \mid i=1,\ldots,4\}$, say $X$, so the leaves are covered by two $X$-blocks.
  Lemma~\ref{lm-2-blocks} implies that $G \in \A_{xy}^1$, a contradiction.

  \subcase{$H$ has precisely three leaves}
  By~(S\ref{it-single-leaf}), there is a single $C$-bridge set $B_v$ that contains all three leaves. 
  Suppose $C$ is a triangle.
  By~(S\ref{it-triangle-att}),
  both vertices of $C$ different from $v$ have both labels and $H$ contains \gph{Star} as a minor.
  
  Suppose $C$ has length at least 4.
  By~(S\ref{it-labels-cycle}), $C-v$ contains the label sequence $XYX$ or $YXY$.
  Thus $H$ has \gph{Tripod} as a minor.

  \subcase{$H$ has precisely two leaves}
  By~(S\ref{it-single-leaf}), there is a single $C$-bridge set $B_v$ that contains both leaves. 
  Let $u$ be a vertex of degree 3 in $B_v - v$ if there is one and let $u=v$ otherwise.
  Let $P$ be the path from $u$ to $v$, possibly of zero length.
  
  Suppose $C$ is a triangle.
  Again by~(S\ref{it-triangle-att}),
  both vertices of $C$ different from $v$ have both labels.
  If $P$ contains both labels, then $H$ has \gph{Alien} as a minor (by~(S\ref{it-leaf})).
  Thus $P$ contains at most one label, say $X$, and then labels $Y$ are covered by two $Y$-blocks, one at the leaves
  and one on the triangle. By Lemma~\ref{lm-2-blocks}, $G \in \A_{xy}^1$, a contradiction.
  
  Suppose $C$ has length at least 4.
  If all vertices in $C - v$ have both labels, then $H$ has \gph{Four} as a minor.
  If $C - v$ contains the label sequence $XYXY$, then $H$ has \gph{Five} as a minor.
  Otherwise, (S\ref{it-labels-cycle}) implies that $C$ has length 4 and
  $C - v$ form the label sequence $YXY$ or $XYX$, say the former.
  If $P$ contains $X$, then $H$ has \gph{Human} as a minor.
  Otherwise, the labels $X$ are covered by two $X$-blocks, one at the leaves and one covering the label $X$ at $C$ --- a contradiction by Lemma~\ref{lm-2-blocks}.

  \subcase{$H$ has precisely 1 leaf}
  Let $w$ be this leaf and $u$ its neighbor.
  By~(S\ref{it-leaf}) and~(S\ref{it-degree}),
  $u$ is unlabelled vertex of degree at least 3 and thus lies on $C$.
  If $C$ has length at most 5, then $H$ contains five vertices with both labels, by Lemma~\ref{lm-alt-boundary}.
  Thus $H$ is isomorphic to \gph{Lollipop}.
  If $C$ has length at least 6 (and, then $H$ has \gph{Mirror} as a minor, by~(S\ref{it-labels-cycle}).

  \case{$H$ has (at least) two cycles, $C_1$ and $C_2$}
  Pick $C_1$ and $C_2$ such that, first, the distance between them is maximal and, second,
  their size is maximal.
  By~(S\ref{it-structure}), $C_1$ and $C_2$ are blocks of $H$ that share at most one vertex.
  Let $P$ be a shortest path (possibly of zero length) joining vertices $v_1 \in V(C_1)$ and $v_2 \in V(C_2)$.
  Note that by the choice of $C_1$ and $C_2$, all $C_1$-bridges attached to $C_1 - v_1$ and 
  all $C_2$-bridges attached to $C_2 - v_2$ are trees.

  \subcase{$C_1$ and $C_2$ are triangles}
  Suppose there is more than one $C_1$-bridge at $v_1$ and let $B$ be a $C_1$-bridge at $v_1$ not containing $P$.
  By~(S\ref{it-endblock}), $B$ contains both labels.
  By~(S\ref{it-triangle-att}), all vertices of $C_1 - v_1$ and $C_2 - v_2$ have both labels attached.
  Thus $H$ has \gph{Star} as a minor.
  So we may assume that there is only one $C_1$-bridge attached at $v_1$.
  Similarly, there is only one $C_2$-bridge attached at $v_2$.

  If there is a $C_1$-bridge attached to a vertex $v$ of $C_1 - v_1$, then
  the $C_1$-bridge set at $v$ is a tree containing at least two leaves by~(S\ref{it-single-leaf}).
  This implies that $H$ has \gph{Star} as a minor. Thus there are no $C_1$-bridges attached to $C_1 - v_1$.
  The same holds for $C_2$ by symmetry.

  If the component $M$ of $H - E(C_1) - E(C_2)$ containing $P$ has both labels, then $H$ has \gph{Bowtie} as a minor.  
  Suppose to the contrary that $M$ has at most one label, say $X$.
  Since there are no other bridges attached to $C_1$ and $C_2$, 
  the $Y$-labelled vertices of $H$ are covered by two $Y$-blocks, a contradiction by Lemma~\ref{lm-2-blocks}. 


  \subcase{$C_1$ is a triangle and $C_2$ has length at least 4}
  If $H$ contains four leaves, then it is not difficult to check that $H$ has \gph{Star} as a minor.
  Hence there is at most one non-trivial bridge set attached to $C_1 - v_1$ or $C_2 - v_2$ (by~(S\ref{it-single-leaf})).
  Suppose that there is a $C_2$-bridge set $B$ attached to a vertex $v$ in $C_2 - v_2$.
  By~(S\ref{it-single-leaf}), $B$ contains at least two leaves. If $B$ contains three leaves, then $H$ has \gph{Star} as minor.
  Therefore, $B$ has precisely two leaves $w_1, w_2$.
  Let $M$ be the component of $H - E(C_1) - w_1 - w_2$ containing $P$.
  By using~(S\ref{it-leaf}), it is easy to see that, if $M$ contains both labels, then $H$ has \gph{Alien} as a minor.
  Otherwise, $M$ has at most one label, say $X$.
  Thus labels $Y$ are covered by two $Y$-blocks, one at $C_1 - v_1$ and the other at $w_1, w_2$.
  A contradiction by Lemma~\ref{lm-2-blocks}.

  Therefore, there is no $C_2$-bridge attached to $C_2 - v_2$.
  By~(S\ref{it-labels-cycle}), $C_2 - v_2$ either contains the sequence $YXY$ or $XYX$, say the former.
  Suppose there is a $C_1$-bridge $B$ attached at $C_1 - v_1$.
  By~(S\ref{it-single-leaf}), $B$ has at least two leaves.
  Hence $H$ has \gph{Tripod} as a minor. Therefore, there is no $C_1$-bridge attached at $C_1 - v_1$
  and both vertices in $C_1 - v_1$ have both labels.
    
  Let $M$ be the component of $H - E(C_1) - E(C_2)$ containing $P$.
  If $M$ contains $X$, then $H$ has \gph{Doll} as a minor.
  Hence $M$ contains at most one label, $Y$.
  If $C_2$ has length at least 5, then $C_2 - v_2$ contains the label sequence $XYXY$ by~(S\ref{it-labels-cycle}).
  It follows that $H$ has \gph{Five} as a minor. Thus $C_2$ has length 4.
  If all vertices in $C_2 - v_2$ contain both labels, then $H$ has \gph{Four} as a minor.
  Otherwise, the labels $X$ at $C_2 - v_2$ can be covered by an $X$-block.
  Since all other labels $X$ are at $C_1 - v_1$ covered by one $X$-block, Lemma~\ref{lm-2-blocks} implies that
  $G \in \A_{xy}^1$, a contradiction.

  \subcase{Both $C_1$ and $C_2$ have length at least 4}
  By~(S\ref{it-single-leaf}), every bridge set attached to $C_1 - v_1$ and $C_2 - v_2$
  contains at least two leaves. Suppose there are non-trivial bridge sets $B_1$ attached to $C_1 - v_1$
  and $B_2$ attached to $C_2 - v_2$, respectively. Since $H$ contains at most four leaves, $B_1$ contains two leaves $w_1, w_2$
  and $B_2$ contains two leaves $w_3, w_4$. If $M = H - \{w_i \mid i=1,\ldots,4\}$ contains both labels,
  then $H$ has \gph{Saddle} as a minor. Otherwise, $M$ has at most one label, say $X$.
  Hence all $Y$ labels are at the leaves and can be covered by two $Y$-blocks.
  A contradiction by Lemma~\ref{lm-2-blocks}.
  If there are two non-trivial bridge sets $B_1, B_2$ attached to one of the cycles, say to $C_1$, then
  $B_1$ and $B_2$ contain together four leaves. By~(S\ref{it-endblock}), there are both labels attached to a vertex of $C_2 - v_2$.
  Hence $H$ has \gph{Star} as a minor.

  Therefore, there is at most one non-trivial bridge set attached to $C_1 - v_1$ and $C_2 - v_2$.
  Suppose there is a $C_1$-bridge set $B$ attached to a vertex $v$ in $C_1 - v_1$.
  By~(S\ref{it-labels-cycle}), $C_2 - v_2$ contains the label sequence $YXY$ or $XYX$, say the former.
  If $B$ contains at least three leaves, then $H$ has \gph{Tripod} as a minor.
  By~(S\ref{it-single-leaf}), $B$ has precisely two leaves $w_1, w_2$.
  If $C_2$ has length at least 5, then 
  $C_2$ contains the sequence $XYXY$, by~(S\ref{it-labels-cycle}).
  Hence $H$ has \gph{Five} as a minor.
  If $C_2$ contains three vertices with both labels, then $H$ has \gph{Four} as a minor.
  If $H - w_1 - w_2 - (C_2 - v_2)$ contains label $X$, then $H$ has \gph{Human} as a minor.
  Otherwise, the $X$ labels at $C_2 - v_2$ can be covered by a single $X$ block
  and all other $X$ labels are at $w_1, w_2$ which are covered by a second $X$ block.
  By Lemma~\ref{lm-2-blocks}, $H$ is 2-alternating, a contradiction.

  By symmetry of $C_1$ and $C_2$, we conclude that there are no non-trivial bridge sets
  attached to $C_1 - v_1$ and $C_2 - v_2$. 
  By~(S\ref{it-labels-cycle}), $C_2 - v_2$ contains the label sequence $YXY$ or $XYX$, say the former.
  By~(S\ref{it-labels-cycle}), $C_1 - v_1$ contains the label sequence $YXY$ or $XYX$.
  If $C_1 - v_1$ contains the sequence $XYX$, then $H$ has \gph{Pinch} as a minor.
  Thus $C_1 - v_1$ contains the sequence $YXY$.
  If $C_2 - v_2$ contains the sequence $XYX$, then $H$ has \gph{Pinch} as a minor.
  Let $M$ be the component of $H - E(C_1) - E(C_2)$
  that contains $P$.
  If $M$ contains label $X$, then $H$ has \gph{Extra} as a minor.
  Otherwise, the labels $X$ can be covered by two $X$-blocks, one at $C_1 - v_1$ and one at $C_2 - v_2$.
  A contradiction by Lemma~\ref{lm-2-blocks}.
  \end{casesblock}%
\end{proof}

\section{The main theorem}
\label{sc-main}

The previous lemmas give rise to the following theorem.

\begin{theorem}
\label{th-main}
  Let $G$ be a graph in $\F_{xy}^1$. Then one of the following holds:
  \begin{enumerate}[\rm(i)]
  \item 
    $G$ is a split of a Kuratowski graph with $x$ and $y$ being the two vertices resulting after the split (see Fig.~\ref{fg-split})
    or $G$ is a Kuratowski graph plus one or two isolated vertices that are terminals.
  \item
    $G$ is an $xy$-sum of two Kuratowski graphs (see Fig.~\ref{fg-kuratowski-sum}).
  \item
    $G$ corresponds to one of the $XY$-labelled graphs in Fig.~\ref{fg-alt-xy},~\ref{fg-alt-2-con}, or~\ref{fg-alt-1-con}.
  \end{enumerate}
\end{theorem}

\begin{proof}
  By Lemma~\ref{lm-k-graph}, either (i) holds or $G /xy$ is planar.
  In the latter case, let $H$ be the $XY$-labelled graph that corresponds to $G$.
  We will now show that $H$ contains one of these graphs as a minor.
  If $H$ is disconnected, then (ii) holds by Lemma~\ref{lm-alt-disconnected}.
  If $H$ is 2-connected, then $H$ is one of the graphs in Fig.~\ref{fg-alt-xy} or~\ref{fg-alt-2-con} by Lemma~\ref{lm-alt-2-con}.
  Otherwise, $H$ is one of the graphs in Fig.~\ref{fg-alt-1-con} by Lemma~\ref{lm-alt-1-con}.

  It is easy to see that none of the graphs in (i)--(iii) contains another one as a minor. Thus, in order to prove that each
  of them is in $\F_{xy}^1$, it suffices to see that they are not in $\A_{xy}^1$. This is clear for (i) since the graphs in (i) are
  non-planar after identifying $x$ and $y$. Similarly, graphs in (ii) cannot be in $\A_{xy}^1$ since they do not have an embedding 
  in the projective plane. Finally, graphs in (iii) are not in $\A_{xy}^1$ since their corresponding $XY$-labelled graphs are not 2-alternating.
\end{proof}

Note that the edge $xy$ is present in a graph $G \in \F_{xy}^1$ if and only if $G - xy$ is planar.
There are only five graphs in $\F_{xy}^1$ with the edge $xy$, the three splits of Kuratowski graphs (see Fig.~\ref{fg-split}) and 
the two graphs in Fig.~\ref{fg-alt-xy}.

\begin{figure}
  \centering
  \includegraphics{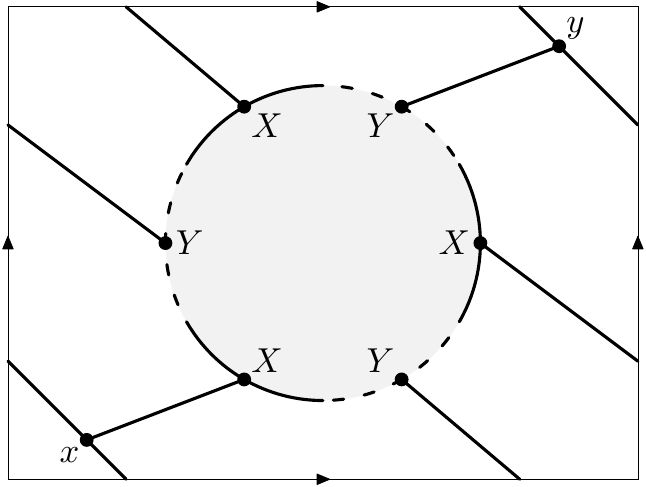}
  \caption{An embedding of a 3-alternating graph in the torus.}
  \label{fg-3-alt}
\end{figure}

\begin{corollary}
  All graphs in $\F_{xy}^1$ embed into the torus. 
\end{corollary}

\begin{proof}
  By Theorem~\ref{th-main}, graphs in $\F_{xy}^1$ are of three types, (i)--(iii).
  For graphs in (i) and (ii), embeddings in the torus are easily constructed.
  The graphs in (iii) are 3-alternating and thus have a planar embedding with three
  $X$-blocks covering the $X$-labels. This embedding can be extended to an 
  embedding in the torus by adding a single handle; 
  see Fig.~\ref{fg-3-alt} where the $X$-blocks are shown as thick intervals on the boundary of the planar part 
  (and $Y$-blocks are shown by thick broken line).
\end{proof}

\bibliographystyle{abbrv}
\bibliography{bibliography}
\end{document}